\documentclass[review,10pt]{elsarticle}
\usepackage{srcltx}
\usepackage{eurosym}
\usepackage{multirow}
\usepackage{mathtools}
\usepackage{amsmath,mathtools}
\usepackage{amsfonts}
\usepackage{amssymb}
\usepackage{amsthm}
\usepackage{graphicx}
\usepackage{mathrsfs}
\usepackage{xcolor}
\usepackage{exscale}
\usepackage{latexsym}
\usepackage{multicol}

\usepackage[backref=page,colorlinks,plainpages=true,pdfpagelabels,
hypertexnames=true,colorlinks=true,pdfstartview=FitV,linkcolor=blue,
citecolor=red,urlcolor=black]{hyperref}
\PassOptionsToPackage{unicode}{hyperref}
\PassOptionsToPackage{naturalnames}{hyperref}
\usepackage{enumerate}
\usepackage[shortlabels]{enumitem}
\usepackage{bookmark}
\usepackage{wasysym}
\usepackage{esint}
\usepackage[ddmmyyyy]{datetime}
\usepackage[margin=2cm]{geometry}
\makeatletter
\g@addto@macro\normalsize{%
  \setlength\abovedisplayskip{2pt}
  \setlength\belowdisplayskip{2pt}
  \setlength\abovedisplayshortskip{4pt}
  \setlength\belowdisplayshortskip{4pt}
}
\numberwithin{equation}{section}
\everymath{\displaystyle}
\usepackage[capitalize,nameinlink]{cleveref}
\crefname{section}{Section}{Sections}
\crefname{subsection}{Subsection}{Subsections}
\crefname{condition}{Condition}{Conditions}
\crefname{hypothesis}{Hypothesis}{Hypothesis}
\crefname{assumption}{Assumption}{Assumptions}
\crefname{lemma}{Lemma}{Lemmas}
\crefname{claim}{Claim}{Claims}
\crefname{remark}{Remark}{Remarks}

\crefformat{equation}{\textup{#2(#1)#3}}
\crefrangeformat{equation}{\textup{#3(#1)#4--#5(#2)#6}}
\crefmultiformat{equation}{\textup{#2(#1)#3}}{ and \textup{#2(#1)#3}}
{, \textup{#2(#1)#3}}{, and \textup{#2(#1)#3}}
\crefrangemultiformat{equation}{\textup{#3(#1)#4--#5(#2)#6}}%
{ and \textup{#3(#1)#4--#5(#2)#6}}{, \textup{#3(#1)#4--#5(#2)#6}}%
{, and \textup{#3(#1)#4--#5(#2)#6}}

\Crefformat{equation}{#2Equation \textup{(#1)}#3}
\Crefrangeformat{equation}{Equations \textup{#3(#1)#4--#5(#2)#6}}
\Crefmultiformat{equation}{Equations \textup{#2(#1)#3}}{ and \textup{#2(#1)#3}}
{, \textup{#2(#1)#3}}{, and \textup{#2(#1)#3}}
\Crefrangemultiformat{equation}{Equations \textup{#3(#1)#4--#5(#2)#6}}%
{ and \textup{#3(#1)#4--#5(#2)#6}}{, \textup{#3(#1)#4--#5(#2)#6}}%
{, and \textup{#3(#1)#4--#5(#2)#6}}

\crefdefaultlabelformat{#2\textup{#1}#3}
\newtheorem{theorem}{Theorem}[section]
\newtheorem{lemma}[theorem]{Lemma}

\newtheorem{proposition}[theorem]{Proposition}

\newtheorem{definition}[theorem]{Definition}
\newtheorem{remark}[theorem]{Remark}        
\newtheorem{assumption}[theorem]{Assumption}  

\numberwithin{equation}{section}

\makeatletter
\newcommand{\vo}{\vec{o}\@ifnextchar{^}{\,}{}}
\makeatother

\def\YYint#1#2#3{{\setbox0=\hbox{$#1{#2#3}{\iint}$}
    \vcenter{\hbox{$#2#3$}}\kern-.50\wd0}}

\def\XXint#1#2#3{{\setbox0=\hbox{$#1{#2#3}{\int}$}
    \vcenter{\hbox{$#2#3$}}\kern-.50\wd0}}

\makeatletter
\def\namedlabel#1#2{\begingroup
   \def\@currentlabel{#2}%
   \label{#1}\endgroup
}
\makeatother
\makeatletter
\newcommand{\rmh}[1]{\mathpalette{\raisem@th{#1}}}
\newcommand{\raisem@th}[3]{\hspace*{-1pt}\raisebox{#1}{$#2#3$}}
\makeatother

\newcommand{\descref}[2]{\hyperref[#1]{\textnormal{\textcolor{black}{(}\textcolor{blue}{\bf #2}\textcolor{black}{)}}}}

\newcommand{\dref}[2]{\hyperref[#1]{\textcolor{black}{(}\textcolor{blue}{\bf #2}\textcolor{black}{)}}}

\newcommand\RR{\mathbb{R}}

\newcommand{\al}{\alpha}

\newcommand{\ve}{\varepsilon}

\newcommand{\jinin}{j\in\{1,2,\ldots,N\}}

\newcommand{\Om}{\Omega}

\DeclareMathOperator{\spt}{spt}

\DeclareMathOperator{\diam}{diam}
\DeclareMathOperator{\dist}{dist}

\DeclareMathOperator{\loc}{loc}

\newcommand{\iprod}[2]{\langle #1 ,  #2\rangle}

\newcommand{\lbr}[1][(]{\left#1}
\newcommand{\rbr}[1][)]{\right#1}

\newcommand{\txt}[1]{\qquad \text{#1} \qquad}

\newcounter{whitney}
\refstepcounter{whitney}

\newcounter{ineqcounter}
\refstepcounter{ineqcounter}

\makeatletter
\def\ps@pprintTitle{%
\let\@oddhead\@empty
\let\@evenhead\@empty
\def\@oddfoot{}%
\let\@evenfoot\@oddfoot}
\makeatother
\usepackage{setspace}
\usepackage[titletoc,toc,page]{appendix}
\definecolor{mycolor}{rgb}{0.122, 0.435, 0.698}

\allowdisplaybreaks
\usepackage{pgf,tikz}
\usetikzlibrary{arrows,patterns}
\usetikzlibrary{decorations.pathreplacing}
\definecolor{babyblue}{rgb}{0.47, 0.27, 0.23}
\nonstopmode
\begin{document}

\begin{frontmatter}

\title{On Lipschitz regularity for bounded minimizers of functionals with $(p,q)$-growth}

\author[myaddress]{Karthik Adimurthi\tnoteref{thanksfirstauthor}}
\corref{mycorrespondingauthor}
\ead{karthikaditi@gmail.com and kadimurthi@tifrbng.res.in}

\author[myaddress]{Vivek Tewary\tnoteref{thankssecondauthor}}
\ead{vivek2020@tifrbng.res.in and vivektewary@protonmail.com}

\tnotetext[thanksfirstauthor]{Supported by the Department of Atomic Energy,  Government of India, under project no.  12-R\&D-TFR-5.01-0520 and SERB grant SRG/2020/000081}
\tnotetext[thankssecondauthor]{Supported by the Department of Atomic Energy,  Government of India, under project no.  12-R\&D-TFR-5.01-0520}
\cortext[mycorrespondingauthor]{Corresponding author}
\address[myaddress]{Tata Institute of Fundamental Research, Centre for Applicable Mathematics, Bangalore, Karnataka, 560065, India}

\begin{abstract}
We obtain Lipschitz estimates for bounded minimizers of functionals with nonstandard $(p,q)$-growth satisfying the dimension-independent restriction $q<p+2$ with $p \geq 2$. This relation improves existing  restrictions when $p \leq N-1$, moreover our result is sharp in the range $N > \frac{p(2+p)}{2} + 1$. The standard Lipschitz regularity takes the form $W^{1,\infty}_{\loc} - W^{1,p}_{\loc}$, whereas we obtain $W^{1,\infty}_{\loc} - L^{\infty}_{\loc}$ regularity estimate and then make use of existing sharp $L^{\infty}_{\loc}$ bounds to obtain the required conclusion.
\end{abstract}


\end{frontmatter}
\begin{singlespace}
\tableofcontents
\end{singlespace}
\section{Introduction}\label{section0}

Let $\Om\subset\RR^N$ with $N \geq 2$ be a bounded open set and we consider the problem of local regularity of minimizers of 
\begin{equation}\label{functional}
    \mathfrak{F}[u] := \int_{\Om} f(\nabla u) \ dx,
\end{equation}
where $f : \RR^N \mapsto \RR$ is a $C^2$ integrand satisfying $(p,q)$ growth of the form 
\begin{assumption}\label{assumption1}
    Let $2\leq p\leq q<\infty$ and suppose there exist constants $m,M\in(0,\infty)$ such that for  any $z\in\RR^N$ and $\xi\in\RR^N$, the following holds   
\begin{subequations}
\begin{gather}
	m|z|^p\leq f(z)\leq M(1+|z|)^q,\label{hyp1}\\
	m|z|^p\leq \iprod{Df(z)}{z} \txt{and} |Df(z)|\leq M(1+|z|^{q-1}),\label{hyp2}\\
	m|z|^{p-2}|\xi|^2\leq \langle D^2f(z)\xi,\xi\rangle\leq M(1+|z|^{q-2})|\xi|^2.\label{hyp3}
\end{gather}
\end{subequations}
\end{assumption}

\begin{definition}\label{mainprob}
    We say that $U\in W^{1,p}_{\loc}(\Om)$ is a local minimizer of \cref{functional} provided the following two conditions are satisfied:
    \begin{enumerate}[(i)]
        \item $f(\nabla U) \in L^1_{\loc}(\Om)$ and
        \item $\int_{\spt \varphi} f(\nabla u) \ dx \leq \int_{\spt \varphi} f(\nabla u + \nabla \varphi) \ dx$ holds for all $\varphi \in W^{1,p}(\Om)$ with $\spt(\varphi) \Subset \Om$.
    \end{enumerate}

\end{definition}


Our main theorem reads as follows:
\begin{theorem}\label{maintheorem} Let $2\leq p\leq q<\infty$ with $q<p+2$ and  let $U\in W^{1,p}_{\loc}(\Om)\cap L^\infty_{\loc}(\Om)$ be a bounded, local minimizer of $\mathfrak{F}$ as in \cref{mainprob},  then $\nabla U\in L^\infty_{\loc}(\Om)$.
\end{theorem}


\subsection{Comparision to Previous Results}
Regularity theory of variational problems with nonstandard $(p,q)$ growth was pioneered by P.Marcellini in a series of seminal papers \cite{marcelliniRegularityMinimizersIntegrals1989,marcelliniRegularityExistenceSolutions1991,marcelliniRegularityEllipticEquations1993,marcelliniRegularityScalarVariational1996,marcelliniEverywhereRegularityClass1996}. Since we are interested in Lipschitz regularity, let us recall that P.Marcellini proves Lipschitz bounds for $U\in W^{1,p}_{\loc}(\Om)$ under the restriction
\begin{equation*}
	\frac{q}{p}<1+\frac{2}{N}.
\end{equation*}
In a recent paper, P.Bella and M.Sch\"affner \cite{bellaRegularityMinimizersScalar2020} improved the restriction to
\begin{align}\label{bellaschaffner}
	\frac{q}{p}<1+\min\left\{1,\frac{2}{N-1}\right\}, \quad \mbox{ for } N\geq 2 \quad \text{and} \quad p \geq 2,
\end{align} by employing a specialized test function that enables them to use Sobolev embedding on the sphere. 
There is a large body of work dealing with problems of $(p,q)$-growth as well as other nonstandard growth problems, for which we refer to \cite{espositoHigherIntegrabilityMinimizers1999,espositoRegularityMinimizersFunctionals1999,espositoRegularityResultsMinimizers2002,espositoSharpRegularityFunctionals2004,colomboRegularityDoublePhase2015,colomboBoundedMinimisersDouble2015,cupiniLocalBoundednessMinimizers2015,cupiniRegularityMinimizersLimit2017,baroniRegularityGeneralFunctionals2018,defilippisRegularityMultiphaseVariational2019,beckLipschitzBoundsNonuniform2020,eleuteriRegularityScalarIntegrals2020,defilippisLipschitzBoundsNonautonomous2020,defilippisRegularityMinimaNonautonomous2020,defilippisInterpolativeGapBounds2021}. A  more detailed survey on the state of the art for problems with nonstandard growth may be found in \cite{marcelliniRegularityGeneralGrowth2020,mingioneRecentDevelopmentsProblems2021}.

It is well known that Lipschitz continuity and even boundedness for \cref{functional} fail when $p$ and $q$ are far apart as evidenced by the following example of Hong \cite{hongRemarksMinimizersVariational1992}, which is a variation on the famous counterexample of Giaquinta \cite{giaquintaGrowthConditionsRegularity1987}:
\begin{align*}
	\int_{\Om} |\nabla u|^2+|u_{x_n}|^4\,dx,
\end{align*} which satisfies \cref{hyp1,hyp2,hyp3} for $p=2$ and $q=4$ and admits an unbounded minimizer if $N\geq 6$ (more examples of unbounded minimizers of \cref{functional} may be found in \cite{marcelliniExempleSolutionDiscontinue1987}). It was shown in \cite[Section~6]{marcelliniRegularityExistenceSolutions1991} that if $q>\frac{(N-1)p}{N-1-p}$, then one cannot expect boundedness and only recently, this restriction was found to be sharp in \cite{hirschGrowthConditionsRegularity2020}, where it is proved that the minimizer is bounded provided
\begin{align}\label{hirschschaffner}
	\frac{1}{p}-\frac{1}{q}\leq \frac{1}{N-1}.
\end{align}

It is easy to see that there is a  gap between the restrictions in \cref{hirschschaffner} and \cref{bellaschaffner} and in this context, the authors in \cite{bellaRegularityMinimizersScalar2020,mingioneRecentDevelopmentsProblems2021} asked if one could obtain a Sobolev-type restriction (as in \cref{hirschschaffner}) in order for the minimizer to be  Lipschitz regular. \emph{In this regard, we improve the restriction in \cref{bellaschaffner} in some special ranges of $p,q$ and $N$ and also partially provide an answer to the question from \cite{bellaRegularityMinimizersScalar2020,mingioneRecentDevelopmentsProblems2021} by obtaining a Sobolev type restriction when $N > \frac{p(2+p)}{2} + 1$.}
\begin{enumerate}[(i)]
    \item For bounded minimizers, we require $q < p+2$, see \cref{maintheorem}.
    \item Combining the restriction $q < p+2$ with the optimal restriction for boundedness from \cref{hirschschaffner}, we see that Lipschitz regularity for minimizers holds provided  $\frac{q}{p} < \min\left\{1+\frac{p}{N-1-p},1+\frac{2}{p} \right\}$.
    \item In the case $p\leq N-1$, we see that $\frac{2}{p} \geq \frac{2}{N-1}$, which suggests that \cref{maintheorem} improves the restriction given in \cref{bellaschaffner}. But it must be noted that our result additionally requires that the solutions are bounded which also requires the restriction \cref{hirschschaffner} to be satisfied. We now compare the two results in a few special cases as follows:
    
     \begin{table}[!htbp] \centering
 \caption{ Admissible values of $q$, shaded regions denote sharp restrictions. }
 \label{}
 \begin{tabular}{|c||c|c|c|c|c|c|c|} 
 \hline 
  & $N=2$& $N=3$& $N=4$& $N=5$ & $N=6$ & $N=7$ &  \\
 \hline \hline\\[-1.8ex]
 \multirow{ 3 }{*}{ $p=2$ }  &  $q<4$ &  $q<4$  & $q<4$   & {$q<4$} & \colorbox{babyblue!20}{$q\leq \frac{10}3$} & \colorbox{babyblue!20}{$q\leq 3 $} & \text{\cref{maintheorem} + \cref{hirschschaffner}} - $C^{0,1}$\\ \cline{2-8} \\[-1.8ex]
   &  $q<4$ &  $q<4$  & $q<\frac{10}3$   &  $q<3$ & $q< \frac{14}5$ & $q<\frac83 $ & \cref{bellaschaffner} - $C^{0,1}$\\ \cline{2-8} \\ [-1.8ex]
   &  \colorbox{babyblue!20}{$q \in (1,\infty]$} &  \colorbox{babyblue!20}{$q\in (1,\infty)$}  & \colorbox{babyblue!20}{$q\leq 6$}   &  \colorbox{babyblue!20}{$q\leq 4$} & \colorbox{babyblue!20}{$q\leq\frac{10}{3}$} & \colorbox{babyblue!20}{$q\leq 3$} & \cref{hirschschaffner} - $L^{\infty}$\\
   \hline \hline\\
 \multirow{ 3 }{*}{ $p=3$ }  &  $q<5$ &  $q<5$  & $q<5$   &  {$q<5$} & {$q<5$} & {$q<5 $} & \text{\cref{maintheorem} + \cref{hirschschaffner}} - $C^{0,1}$\\ \cline{2-8} \\[-1.8ex]
   &  $q<6$ &  $q<6$  & $q<5$   &  $q<\frac{9}{2}$ & $q< \frac{21}5$ & $q<4 $ & \cref{bellaschaffner} - $C^{0,1}$\\ \cline{2-8} \\ [-1.8ex]
   &  \colorbox{babyblue!20}{$q \in (1,\infty]$} &  \colorbox{babyblue!20}{$q\in (1,\infty]$}  & \colorbox{babyblue!20}{$q\in (1,\infty)$}   &  \colorbox{babyblue!20}{$q\leq 12$} & \colorbox{babyblue!20}{$q\leq\frac{15}{2}$} & \colorbox{babyblue!20}{$q\leq 6$} & \cref{hirschschaffner} - $L^{\infty}$\\
   \hline \hline\\
 \multirow{ 3 }{*}{ $p=4$ }  &  $q<6$ &  $q<6$  & $q<6$   &  {$q<6$} & {$q<6$} & {$q<6$} & \text{\cref{maintheorem} + \cref{hirschschaffner}} - $C^{0,1}$\\ \cline{2-8} \\[-1.8ex]
   &  $q<8$ &  $q<8$  & $q<\frac{20}3$   &  $q<6$ & $q< \frac{28}5$ & $q<\frac{16}3 $ & \cref{bellaschaffner} - $C^{0,1}$\\ \cline{2-8} \\ [-1.8ex]
   &  \colorbox{babyblue!20}{$q \in (1,\infty]$} &  \colorbox{babyblue!20}{$q\in (1,\infty]$}  & \colorbox{babyblue!20}{$q \in (1,\infty]$}   &  \colorbox{babyblue!20}{$q \in (1,\infty)$} & \colorbox{babyblue!20}{$q\leq 20$} & \colorbox{babyblue!20}{$q\leq 12$} & \cref{hirschschaffner} - $L^{\infty}$\\
   \hline
 \end{tabular}
 \end{table}

    \item Since we require bounded solutions, we see that for minimizers, Lipschitz regularity would then require $\frac{q}{p} < 1+\min\left\{\frac{p}{N-1-p},\frac{2}{p} \right\}$. In particular, if $N >  \frac{p(p+2)}{2}+1$, then $\frac{2}{p} > \frac{p}{N-1-p}$  and thus Lipschitz regularity holds for any minimizer as they are automatically bounded. In particular, due to the sharpness of the condition \cref{hirschschaffner}, we automatically obtain sharpness of the Lipschitz regularity in this range.
    \item Our theorem improves the previous restriction for bounded minimizers of \cref{functional} which was found to be $q<p+1$ in \cite{choeInteriorBehaviourMinimizers1992,cupiniExistenceRegularityElliptic2014}. 
\end{enumerate}

We now briefly describe the method of proof, first, we begin with a regularization procedure following \cite{bousquetLipschitzRegularityOrthotropic2020} where a quadratic term is added to $f$. The regularized solution is shown to be in $C^{1,\gamma}\cap W^{2,2}$, and we exploit this $W^{2,2}$ regularity to obtain a gradient higher integrability result.  After obtaining a Caccioppoli-type inequality for the gradient of $U$, we prove that $\nabla U\in L^s_{\loc}$ for all $s\in(1,\infty)$. In fact, $\|\nabla U\|_s$ is estimated in terms of $\|U\|_{L^{\infty}}$ provided $q < p+2$ holds. Finally, we use a Moser iteration adapted for solutions of equations with unbalanced growth to obtain the required result.

\begin{remark}
	After this paper was written, we became aware that this result has been proved previously by Bildhauer and Fuchs~\cite{BF2002} in 2002. The methods are, for the most part, similar. Instead of De Giorgi iteration, we use a Moser iteration in the last step. Also, we use a quadratic term for the regularization and we use the notion of bounded slope condition to gain Lipschitz regularity for the regularized minimizer. 
\end{remark}

\section{Notations and Preliminaries}\label{section1}
\subsection{Notations}
We begin by collecting the standard notation that will be used throughout the paper.
\begin{itemize}
	\item We shall denote $N$ to be the space dimension. A point in $\mathbb{R}^{N}$ will be denoted by $x$. 
	\item Let $\Omega$ be a domain in $\mathbb{R}^N$ of boundary $\partial \Omega$.
	\item The notation $a \lesssim b$ is shorthand for $a\leq C b$ where $C$ is a constant independent of the regularization parameters $\sigma$ and $\ve$. 
	\item We will use the symbol $\iprod{\cdot}{\cdot}$ to denote the Euclidean inner product.
\end{itemize}
\subsection{Preliminaries for Regularization}
We list some of the preliminaries that are required in the subsequent sections. The regularization procedure relies on the addition of a quadratic term to the functional. Stampacchia \cite{stampacchiaRegularMultipleIntegral1963} has proved some general theorems on the local Lipschitz regularity of minimizers to convex minimization problems posed on convex domains with boundary values satisfying the {\itshape{bounded slope condition}}.  

\begin{definition}[Uniformly Convex set]
	A bounded, open set $B\subset\mathbb{R}^N$ is said to be uniformly convex if there exists $\nu>0$ such that for every boundary point $x_0\in\partial B$ there exists a hyperplane $H_{x_0}$ passing through that point satisfying
	\begin{align*}
		\dist(y,H_{x_0})\geq \nu |y-x_0|^2 \quad \mbox{ for every $y\in\partial B$.}
	\end{align*}
\end{definition}

\begin{definition}[Bounded Slope Condition]
	Let $K$ be a positive real number and $B$ an open bounded convex subset of $\mathbb{R}^N$. We say that a function $\phi:\partial B\to\mathbb{R}$ satisfies the bounded slope condition of rank $K$ if for any $x_0\in\partial B$ there exists vectors $l_{x_0}^-$ and $l_{x_0}^+$ satisfying $||l_{x_0}^-||\leq K$ and   $||l_{x_0}^+||\leq K$ such that
	\begin{align*}
		\iprod{l_{x_0}^-}{x-x_0}\leq \phi(x)-\phi(x_0)\leq \iprod{l_{x_0}^+}{x-x_0} \quad \mbox{ holds for every $x\in\partial B$.}
	\end{align*}
\end{definition}

The following proposition gives a necessary and sufficient condition for a function to satisfy the bounded slope condition, see \cite[Corollary 4.3]{hartman1966bounded} and \cite{mirandaTeoremaDiEsistenza1965} for the details of the proof.
\begin{proposition}\label{characterbsc}
	Let $B\subset\mathbb{R}^N$ be a bounded, open and  uniformly convex domain with the boundary being $\partial B \in C^{1,1}$ regular. Then, a necessary and sufficient condition for any function $\phi(x)$, $x \in \partial B$  to satisfy bounded slope condition is that $\phi(x) \in C^{1,1}(\partial B)$.
\end{proposition}
%
\begin{theorem}[\cite{stampacchiaRegularMultipleIntegral1963}, Theorem 9.2]\label{lipstampacchia}
	Let $H\in C^2(\mathbb{R}^N)$ and assume it satisfies $\iprod{D^2H(p)\zeta}{\zeta} \geq \nu |\zeta|^2$ for all $\zeta \in \RR^N$. If $\Om\subset\mathbb{R}^N$ is uniformly convex and $C^{1,1}$, then the integral given by $I(u)=\int_{\Om} H(\nabla u)\,dx$ attains its minimum in the class of all Lipschitz functions in $\Om$ assuming that the boundary values satisfy the bounded slope condition and are the trace of a $W^{2,p}$ function for some $p>n$.
\end{theorem}
Let us now recall some basic facts from calculus of variations that will be needed later on. The first concerns the existence of a minimizer, the proof of which can be found in \cite[Theorem 2.7]{rindler}.
\begin{theorem}\label{existence}
Let $f:\mathbb{R}^N\to [0,\infty)$ be a $C^2$ function such that
\begin{itemize}
	\item[(i)] $f$ satisfies the $r$-coercivity bound, i.e.,  $f(z)\geq m |z|^r$ holds for all $z\in\mathbb{R}^N$ and some $r\in(1,\infty)$.
	\item[(ii)] $f$ is a convex function.
\end{itemize}
Then, the associated functional $\mathcal{F}(u)=\int_\Omega f(\nabla u)\,dx$ has a minimizer over $W^{1,r}_g(\Omega)=\{u\in W^{1,r}(\Omega):u|_{\partial\Omega}=g\}$, where $g\in W^{1-1/r,r}(\partial \Omega)$.
\end{theorem}

The second result discusses when does there exist a unique solution, the proof of which can be found in \cite[Proposition 2.10]{rindler}.
\begin{theorem}\label{uniqueness}
	Let $\mathcal{F}:W^{1,r}(\Omega)\to\mathbb{R},\, r\in [1,\infty)$, be an integral functional with a $C^2$ integrand $f:\mathbb{R}^N\to\mathbb{R}$. If f is strictly convex, i.e., 
	\[f(\theta z_1+(1-\theta)z_2)<\theta f(z_1)+(1-\theta)f(z_2),\] holds for all $z_1,z_2\in\mathbb{R}^N$ with $z_1\neq z_2$ and any $\theta\in (0,1)$, then the minimizer $u_*\in W^{1,r}_g(\Omega)=\{u\in W^{1,r}(\Omega):u|_{\partial\Omega}=g\}$ of $\mathcal{F}$, where $g\in W^{1-1/r,r}(\partial \Omega)$, if it exists, is unique.
\end{theorem}

The next theorem gives a criterion for the integrand to be strictly convex, the proof of which can be found in \cite[Theorem 1.5]{simon2011}.
\begin{theorem}\label{strictlyconvex}
	Let $\Omega$ be an open convex subset of $\mathbb{R}^N$ and let $f:\Om\to\mathbb{R}$ be $C^2$.	Suppose that for all $x\in \Om$ the Hessian
matrix $D^2f(x)$ is strictly positive-definite, then $f$ is strictly convex.
\end{theorem}

We end this subsection by recalling a maximum principle, whose proof may be found in \cite[Theorem 2.1]{stampacchiaRegularMultipleIntegral1963}.

\begin{theorem}[Maximum Principle]\label{maxprince}
	Let $a_{ij}(x), i,j=1,2,\ldots,N$ be measurable and bounded functions in $B$ such that 
	\begin{align*}
		a_{ij}\xi_i\xi_j\geq \mu|\xi|^2, \mbox{ a.e. } x\in B, \mbox{ for all }\xi\in\mathbb{R}^N.
	\end{align*} If $u\in H^1(B)$ satisfies
\begin{align*}
	\int_B a_{ij}(x)u_{x_i}(x)v_{x_j}(x)\,dx=0,\mbox{ for all }v\in H^1_0(B),
\end{align*} then we have
\begin{align*}
	u(x)\leq \max_{x\in\partial B} u(x)\mbox{ a.e. }x\in B.
\end{align*}
\end{theorem}


We shall make use of the following well-known iteration lemma whose proof may be found in \cite[Lemma 6.1]{giustiDirectMethodsCalculus2003}.

\begin{lemma}\label{iterlemma}
	Let $Z(t)$ be a bounded non-negative function in the interval $\rho,R$. Assume that for $\rho\leq t<s\leq R$ we have
	\begin{align*}
		Z(t)\leq [A(s-t)^{-\alpha}+B(s-t)^{-\beta}+C]+\vartheta Z(s)
	\end{align*} with $A,B,C\geq 0$, $\alpha,\beta>0$ and $0\leq \vartheta<1$. Then,
\begin{align*}
	Z(\rho)\leq c(\alpha,\vartheta)[A(R-\rho)^{-\alpha}+B(R-\rho)^{-\beta}+C].
\end{align*}
\end{lemma}

\section{Regularization}\label{section2}

\subsection{Approximation Scheme}

Let us fix a ball $B\Subset\Om$ such that $4B\Subset\Om$. Let $\ve_0=\min\left\{1,\frac{\diam(B)}{2}\right\}>0$. For any $\ve\in(0,\ve_0)$, using a standard mollifier $\rho_{\ve}$ supported in a ball of radius $\ve$ centered at the origin, we define $U_{\ve}\coloneqq U\ast\rho_{\ve}$. For $0<\sigma<1$, we define the regularized functional
\begin{align}\label{mainprob2}
	\mathfrak{F}_{\sigma}(w):=\int_{\Om} f_{\sigma}(\nabla w)\,dx:=\int_{\Om} f(\nabla w)+\frac{\sigma}{2}|\nabla w|^2\,dx,
\end{align} where $f_{\sigma}\in C^2(\mathbb{R}^N)$ satisfies the following growth and ellipticity conditions: 
From \cref{assumption1}, we see that  for $2\leq p\leq q<\infty$, $z\in\RR^N$ and $\xi\in\RR^N$, the following is satisfied:
\begin{subequations}
\begin{gather}
	\frac{\sigma}{2}|z|^2+m|z|^p\leq f_{\sigma}(z)\leq M(1+|z|)^q+\frac{\sigma}{2}|z|^2\label{rhyp1}\\
	\frac{\sigma}{2}|z|^2+m|z|^p\leq \iprod{Df_{\sigma}(z)}{z} \txt{and}  |Df_{\sigma}(z)|\leq M(1+|z|^{q-1})\label{rhyp2}\\
	m|z|^{p-2}|\xi|^2+\sigma|\xi|^2\leq \langle D^2f_{\sigma}(z)\xi,\xi\rangle\leq M(1+|z|^{q-2})|\xi|^2+\sigma|\xi|^2.\label{rhyp3}
\end{gather}
\end{subequations}    
An application of \cref{existence} shows that there exists a minimizer $u_{\sigma,\ve}$ of $\mathfrak{F}_{\sigma}$ in $B$, i.e., the following holds:
\begin{align}\label{regminim}
	\mathfrak{F}_{\sigma}(u_{\sigma,\ve})=\min_{v\in U_{\ve}+W^{1,p}_0(B)}\int_{\Om} f(\nabla v)+\frac{\sigma}{2}|\nabla v|^2\,dx.
\end{align} 
From \cref{rhyp3} and \cref{strictlyconvex}, we see that $f_{\sigma}$ is strictly convex and thus, \cref{uniqueness} shows that $u_{\sigma,\ve}$ is unique.

\subsection{Regularity of minimizers}

\begin{lemma}\label{reglemma} The unique minimizer $u_{\sigma,\ve}$ of \cref{mainprob2} belongs to 
	$u_{\sigma,\ve}\in L^{\infty}_{\loc}(B) \cap C^{0,1}_{\loc}(B)\cap W^{2,2}_{\loc}(B)$.
	Moreover, for all $0<\ve<\ve_0$ and $0<\sigma<1$, the following holds:
	$$||u_{\sigma,\ve}||_{L^\infty(B)}\leq||U||_{L^\infty(2B)}.$$
\end{lemma}
\begin{proof}
Since the functional $\mathfrak{F}_{\sigma}$ and boundary data $U_{\ve}$ given in \cref{regminim} satisfies the hypothesis from \cref{lipstampacchia}, we have the existence of a minimizer $u_{\sigma,\ve}$ which is Lipschitz regular. Moreover, the minimizer is unique which follows from \cref{uniqueness} and \cref{rhyp3} applied with \cref{strictlyconvex}.

Clearly the function $u_{\sigma,\ve}$ satisfies the following Euler-Lagrange equation
	\begin{align*}
		\nabla\cdot(\nabla f_{\sigma}(\nabla u_{\sigma,\ve}))=0 \qquad \text{in   } B.	\end{align*}
To prove the bound, we invoke the Maximum principle from \cref{maxprince} along with \cref{rhyp2} to conclude that
\begin{align*}
	\max_{B}|u_{\sigma,\ve}|\leq \max_{\partial B}|u_{\sigma,\ve}|= \max_{\partial B}|U_{\epsilon}|\leq \max_{2B}|U|.
\end{align*}
Thus, we get $u_{\sigma,\ve} \in L^{\infty}(B) \cap  W^{1,\infty}(B)$. Note that the $L^{\infty}$ estimate is uniform and independent of $\sigma$ and $\ve$, whereas the Lipschitz bound could possibly depend on $\sigma$ and $\ve$.

 Noting that the functional $f_{\sigma}$ satisfies \cref{rhyp3},  by a standard argument involving difference quotients, we can prove that $u_{\sigma,\ve}\in W^{2,2}(B)$. Note that the $W^{2,2}(B)$ estimate comes from the regularizing term $\frac{\sigma}{2} |z|^2$ in \cref{rhyp1} and depends on the parameter $\sigma$. In particular, the $W^{2,2}(B)$ estimate could possibly blow up as $\sigma \rightarrow 0$. 
\end{proof}

\begin{remark}\label{remark_supression}
In subsequent sections excepting \cref{section6}, we shall suppress the subscript of $u_{\sigma,\ve}$ for ease of notation.
\end{remark}
\section{Caccioppoli Inequality}\label{section3}

We shall prove the following Caccioppoli inequality for the gradient of $u$. Note that the proof is only formal and everything can be made rigorous using difference quotients and the a priori regularity from \cref{reglemma}.
\begin{proposition}\label{caccioppoli}
	Let $\alpha\geq 0$. Let $u$ be the solution to \cref{regminim}. Then it holds that
	\begin{align}\label{cacc1}
		\int_B |\nabla u|^{p-2+\alpha}|\nabla^2 u|^2\eta^2\,dx\leq C(M,m)\biggr\{ \int_B \left(|\nabla u|^{q+\alpha}+|\nabla u|^{2+\alpha}\right)|\nabla \eta|^2\,dx \biggr\}.
	\end{align}
\end{proposition}

\begin{proof}
	The minimizer $u$ of \cref{regminim} satisfies the following Euler-Lagrange equation:
	\begin{align*}
		\int_B\langle Df_{\sigma}(\nabla u),\nabla \phi\rangle\,dx=0.
	\end{align*} By choosing $\phi=\psi_{x_j}\in H^1_0(B)$ and integrating by parts, we get
	\begin{align*}
	\int_B\langle D^2f_{\sigma}(\nabla u)\nabla u_{x_j},\nabla \psi\rangle\,dx=0.
	\end{align*} Now, for $\kappa> 0$, we choose $\psi=u_{x_j}(\kappa+|\nabla u|^2)^{\frac \alpha 2}\eta^2$, where $\eta\in C_0^\infty(B)$ such that $0\leq \eta\leq 1$ in $B$, to get
	\begin{align}\label{step1}
	\underbrace{\int_B\left\langle D^2f_{\sigma}(\nabla u)\nabla u_{x_j},\nabla u_{x_j}\right\rangle(\kappa+|\nabla u|^2)^{\frac \alpha 2}\eta^2\,dx}_{I}&+\underbrace{\int_B\left\langle D^2f_{\sigma}(\nabla u)\nabla u_{x_j},\nabla \left((\kappa+|\nabla u|^2)^{\frac \alpha 2}\right)\right\rangle u_{x_j}\eta^2\,dx}_{II}\nonumber\\
	&=\underbrace{-2\int_B\left\langle D^2f_{\sigma}(\nabla u)\nabla u_{x_j},\nabla \eta\right\rangle u_{x_j}\eta(\kappa+|\nabla u|^2)^{\frac \alpha 2}\,dx}_{III}.
\end{align}
Now observe that due to the coercivity of $D^2f_{\sigma}(z)$, we can apply the Cauchy-Schwarz inequality for positive definite Hermitian matrix $O$ given by $\langle Ox,y\rangle\leq \langle Ox,x\rangle^{1/2}\langle Oy,y\rangle^{1/2}$ along with Young's inequality to get 
\begin{align}\label{step2}
	III\leq \frac{1}{2}I+2\int_B\langle D^2f_{\sigma}(\nabla u)\nabla \eta,\nabla\eta\rangle u_{x_j}^2(\kappa+|\nabla u|^2)^{\frac \alpha 2}\,dx.
\end{align} Substituting \cref{step2} in \cref{step1} and summing over $\jinin$, we get
\begin{align*}
	\frac{1}{2}\sum_{j=1}^N \int_B\left\langle D^2f_{\sigma}(\nabla u)\nabla u_{x_j},\nabla u_{x_j}\right\rangle(\kappa+|\nabla u|^2)^{\frac \alpha 2}\eta^2\,dx&+\sum_{j=1}^N \int_B\left\langle D^2f_{\sigma}(\nabla u)\nabla u_{x_j},\nabla \left((\kappa+|\nabla u|^2)^{\frac \alpha 2}\right)\right\rangle u_{x_j}\eta^2\,dx\\
	&\leq 2\sum_{j=1}^N\int_B\left\langle D^2f_{\sigma}(\nabla u)\nabla \eta,\nabla\eta\right\rangle u_{x_j}^2(\kappa+|\nabla u|^2)^{\frac \alpha 2}\,dx.
\end{align*}
Observing that $\sum_{j=1}^N u_{x_j}\nabla u_{x_j}=(\kappa+|\nabla u|^2)^{\frac 1 2}\nabla\left( (\kappa+|\nabla u|^2)^{\frac 1 2} \right)$, we can rewrite the previous estimate as 
\begin{align*}
	\frac{1}{2}\sum_{j=1}^N \int_B\big\langle D^2f_{\sigma}&(\nabla u)\nabla u_{x_j},\nabla u_{x_j}\big\rangle(\kappa+|\nabla u|^2)^{\frac \alpha 2}\eta^2\,dx\\
	&+\alpha\int_B\left\langle D^2f_{\sigma}(\nabla u)\nabla \left((\kappa+|\nabla u|^2)^{\frac 1 2}\right),\nabla \left((\kappa+|\nabla u|^2)^{\frac 1 2}\right)\right\rangle (\kappa+|\nabla u|^2)^{\frac \alpha 2}\eta^2\,dx\\
	&\leq 2\sum_{j=1}^N\int_B\left\langle D^2f_{\sigma}(\nabla u)\nabla \eta,\nabla\eta\right\rangle u_{x_j}^2(\kappa+|\nabla u|^2)^{\frac \alpha 2}\,dx.
\end{align*}
Now, we shall apply \cref{rhyp3} to get
\begin{align*}
	\frac{m}{2}\int_B|\nabla u|^{p-2}|\nabla^2 u|^2(\kappa+|\nabla u|^2)^{\frac \alpha 2}\eta^2\,dx&+\alpha m\int_B|\nabla u|^{p-2}\left|\nabla\left((\kappa+|\nabla u|^2)^{\frac 1 2}\right)\right|^2 (\kappa+|\nabla u|^2)^{\frac \alpha 2}\eta^2\,dx\\
	&\leq 2(M+1)\int_B \left(|\nabla u|^q+|\nabla u|^2\right)(\kappa+|\nabla u|^2)^{\frac \alpha 2}|\nabla \eta|^2\,dx.
\end{align*}
We can derive the required inequality by dropping the second term on the left hand side and passing to the limit $\kappa\to0$ by an application of Dominated Convergence Theorem on the right hand side and Fatou's lemma on the left hand side.
\end{proof}


\section{Higher Integrability of gradient}\label{section4} In this section, we will prove that $\nabla u\in L^s_{\loc}(B)$ for all $s\in(1,\infty)$ provided $q<p+2$. 

\begin{proposition}
	Assume that $q<p+2$ and  let $u$ be the unique solution of \cref{regminim} (recall \cref{remark_supression}). For every $\beta\geq 2$, there exists a constant $C=C(N,p,q,\beta,M,m)>0$ such that for every $0<\sigma<1$, $0<\ve<\ve_0$ and every pair of concentric balls $B_{r_0}\Subset B_{R_0}\Subset B$,  we have
	\begin{align}\label{highinteg}
	   \int_{B_{r_0}}|\nabla u|^{p+\beta}\leq &CR_0^N\left[\left(\frac{||U||_{L^\infty(2B_{R_0})}}{R_0-r_0}\right)^{\frac{2(\beta+p)}{p}}+\left(\frac{||U||_{L^\infty(2B_{R_0})}}{R_0-r_0}\right)^{\frac{2(\beta+p)}{2-q+p}}+\left(\frac{||U||_{L^\infty(2B_{R_0})}}{R_0-r_0}\right)^{p+\beta}\right].
	\end{align}
\end{proposition}

\begin{proof}
	We begin with the integral
	\begin{align*}
		\int_B |\nabla u|^{p+\beta}\eta^2\,dx=\int_B \nabla u\cdot(|\nabla u|^{p+\beta-2}\,\nabla u\,\eta^2)\,dx,
	\end{align*}
	where $\eta \in C_c^{\infty}(B)$ to be eventually chosen appropriately. 
	After integrating by parts, we get
		\begin{align*}
		\int_B |\nabla u|^{p+\beta}\eta^2\,dx&=-\int_B  u \nabla\cdot(|\nabla u|^{p+\beta-2}\,\nabla u\,\eta^2)\,dx\\
		&\leq ||u||_{L^\infty(B)}(p+\beta-1)\underbrace{\int_B |\nabla u|^{p+\beta-2}|\nabla^2 u|\eta^2\,dx}_{A_1}+2||u||_{L^\infty(B)}\underbrace{\int_B|\nabla u|^{p+\beta-1}|\eta||\nabla \eta|\,dx}_{A_2}.
	\end{align*}
Applying Young's inequality, we get 
\begin{equation*}
\begin{array}{rcl}
	A_1 &\leq & \tau\int_B |\nabla u|^{p+\beta}\eta^2\,dx+\frac{1}{4\tau} \int_B |\nabla u|^{p+\beta-4}|\nabla^2 u|^2\eta^2\,dx,\\
	A_2 &\leq & \tau \int_B |\nabla u|^{p+\beta}\eta^2\,dx + \frac{1}{4\tau}\int_B |\nabla u|^{p+\beta-2}|\nabla\eta|^2\,dx.
	\end{array}
\end{equation*} 
Choosing $\tau=\frac{1}{4||u||_{L^\infty(B)}(p+\beta-1)}$, we get
\begin{align*}
	\int_B |\nabla u|^{p+\beta}\eta^2\,dx\leq C(p,\beta)||u||^2_{L^\infty(B)}\left\{ \int_B |\nabla u|^{p+\beta-4}|\nabla^2 u|^2\eta^2\,dx+\int_B |\nabla u|^{p+\beta
		-2}|\nabla \eta|^2\,dx \right\}.
\end{align*} For the first integral on the RHS, we will apply \cref{caccioppoli} to get
\begin{align*}
	\int_B |\nabla u|^{p+\beta}\eta^2\,dx\leq C(p,\beta,M,m)||u||^2_{L^\infty(B)}\left\{ \int_B \left(|\nabla u|^{q+\beta-2}+|\nabla u|^{\beta}+|\nabla u|^{p+\beta-2}\right)|\nabla \eta|^2\,dx \right\}.
\end{align*}

Let us now fix a pair of concentric balls $B_r\subset B_R\subset B$ and choose the cut-off function  $\eta\in C_0^\infty(B_R)$ such that $\eta\equiv 1$ on $B_r$ and $|\nabla \eta|\leq \frac{C}{R-r}$, to get
\begin{align}\label{cacc2}
	\int_{B_r} |\nabla u|^{p+\beta}\,dx\leq C(p,\beta,M,m)\frac{||u||^2_{L^\infty(B)}}{(R-r)^2}\left\{ \int_{B_R} \left(|\nabla u|^{q+\beta-2}+|\nabla u|^{\beta}+|\nabla u|^{p+\beta-2}\right)\,dx \right\}.
\end{align} Now, for $\tau>0$ to be chosen, we estimate each of the terms on the right hand side of \cref{cacc2} as follows:
\begin{subequations}
\begin{align}\label{est1}
	|\nabla u|^{p+\beta-2}&\leq \tau |\nabla u|^{p+\beta}+\frac{C(p,\beta)}{\tau^{\frac{p+\beta-2}{2}}},\\
	\label{est2}
	|\nabla u|^\beta&\leq \tau|\nabla u|^{p+\beta}+\frac{C(\beta)}{\tau^{\frac{\beta}{p}}},\\
	\label{est3}
	|\nabla u|^{q+\beta-2}&=|\nabla u|^{q+\beta-2}\frac{\tau^{\frac{q+\beta-2}{p+\beta}}}{\tau^{\frac{q+\beta-2}{p+\beta}}}\leq \tau|\nabla u|^{p+\beta}+\frac{C(p,q,\beta)}{\tau^{\frac{q+\beta-2}{2-q+p}}}.
\end{align}
\end{subequations} For the last inequality \cref{est3}, we apply Young's inequality with exponents $\frac{p+\beta}{q+\beta-2}$ and $\frac{p+\beta}{2+p-q}$, which requires the condition $q<p+2$. Substituting \cref{est1,est2,est3} in \cref{cacc2}, we get
\begin{align*}
	\int_{B_r}|\nabla u|^{p+\beta}\,dx\leq \frac{3C||u||^2_{L^\infty(B)}}{(R-r)^2}\tau\int_{B_R}|\nabla u|^{p+\beta}\,dx
	+\frac{C||u||^2_{L^\infty(B)}}{(R-r)^2}\left\{\left[ \frac{1}{\tau^{\frac{p+\beta-2}{2}}} +  \frac{1}{\tau^{\frac{\beta}{p}}}+\frac{1}{\tau^{\frac{q+\beta-2}{2-q+p}}}\right]R^N\right\},
\end{align*} where $C=C(p,q,M,m,\beta)$. Now, we choose $\tau=\frac{(R-r)^2}{6C||u||_{L^\infty(B)}^2}$ to get
\begin{align*}
	\int_{B_r}|\nabla u|^{p+\beta}\,dx\leq \frac{1}{2}\int_{B_R}|\nabla u|^{p+\beta}\,dx+C R^N\lbr[[] \left(\frac{||u||_{L^\infty(B)}}{R-r}\right)^{p+\beta} + \left(\frac{||u||_{L^\infty(B)}}{R-r}\right)^{\frac{2(\beta+p)}{p+2-q}}+ \left(\frac{||u||_{L^\infty(B)}}{R-r}\right)^{\frac{2(\beta+p)}{p}}\rbr[]].
\end{align*} Now, we fix $r_0<R_0$, then by the iteration \cref{iterlemma} for $r_0\leq r<R\leq R_0$ and the maximum principle in \cref{reglemma}, we obtain \cref{highinteg}.
\end{proof}

\section{Moser's Iteration}\label{section5}
Now, we are in a position to use Moser's iteration for an unbalanced Caccioppoli inequality as in \cref{cacc1}. The difference from a standard Moser's iteration is that the starting point of our iteration must require an exponent of $|\nabla u|$ higher than $p$. In order to do this,  we follow the same scheme as laid out in \cite{choeInteriorBehaviourMinimizers1992}.

\begin{theorem}
	Assume that $q<p+2$ and let $u$ be the solution to \cref{regminim}. Let $R_0>0$ be such that $B_{2R_0}\Subset B$ and we fix the following two exponents: \[
	2^*
	:=
	\begin{cases}
		\frac{2N}{N-2}, & N>2\\
		(2,\infty), & N=2.
	\end{cases} \txt{ and }\alpha_0:=\max\left\{ \frac{2q-{2^*} p}{{2^*}-2},2 \right\}.
	\]Then, there exists $C=C(p,q,N,M,m,R_0,||U||_{L^\infty(2B)})$ such that
	\begin{align}\label{lipbound1}
		||\nabla u||_{L^\infty\left(B_{\frac{R_0}{2}}\right)}\leq C.
	\end{align}
\end{theorem}

\begin{proof}
	We restate the Caccioppoli inequality from \cref{caccioppoli} as:
	\begin{align*}
		\int_B \left|\nabla\left(|\nabla u|^{\frac{p+\alpha}{2}}\right)\right|^2\eta^2\,dx\leq C(M,m)(p+\al)^2\left\{ \int_B\left(|\nabla u|^{q+\alpha}+|\nabla u|^{2+\alpha}\right)|\nabla \eta|^2\,dx \right\},
	\end{align*} which may be further revised to
\begin{align*}
	\int_B \left|\nabla\left(|\nabla u|^{\frac{p+\alpha}{2}}\eta\right)\right|^2\,dx\leq C(M,m)(p+\al)^2\left\{ \int_B\left(|\nabla u|^{q+\alpha}+1\right)|\nabla \eta|^2\,dx \right\}.
\end{align*}
By Sobolev's embedding, we have
\begin{align*}
	\left(\int_B|\nabla u|^{\frac{(p+\alpha){2^*}}{2}}\eta^{{2^*}}\,dx\right)^{{2}/{{2^*}}}\leq C(M,m)(p+\al)^2\left\{ \int_B\left(|\nabla u|^{q+\alpha}+1\right)|\nabla \eta|^2\,dx \right\}.
\end{align*}
Finally, for two concentric balls $B_{\rho}\subset B_{R}$ and $\eta\in C_0^\infty(B_R)$ satisfying $\eta\equiv 1$ in $B_{\rho}$ and $|\nabla \eta|\leq \frac{C}{R-r}$, we have
\begin{align}\label{eqn6.1}
	\left(\int_{B_{\rho}}|\nabla u|^{\frac{(p+\alpha){2^*}}{2}}\,dx\right)^{{2}/{{2^*}}}\leq \frac{C(M,m)(p+\al)^2}{(R-\rho)^2}\left\{ \int_{B_R}\left(|\nabla u|^{q+\alpha}+1\right)\,dx \right\}.
\end{align}

Now, for $n=1,2,3,\ldots$, we define $\rho_n :=\frac{R_0}{2}\left(1+\frac{1}{2^n}\right)$ and choose $\alpha_n$ to satisfy $(p+\alpha_n)\frac{2^*}{2}=\alpha_{n+1}+q$. Therefore, we have
\begin{align}\label{eqn6.2}
	\alpha_{n}=\left(\frac{{2^*}}{2}\right)^n\alpha_0+\left(\frac{{2^*} p}{2}-q\right)\frac{\left(\frac{{2^*}}{2}\right)^n-1}{\frac{{2^*}}{2}-1}.
\end{align}
From \cref{eqn6.2}, it is easy to see that if $\alpha_0+\frac{{2^*} p-2q}{{2^*} -2}>0$, then $\lim_{n\to\infty}\alpha_n=\infty$. Thus, we can rewrite \cref{eqn6.1} as 
\begin{align}\label{eqn6.3}
	\int_{B_{\rho_{n+1}}}|\nabla u|^{\alpha_{n+1}+q}\,dx\leq C^n\left\{ \int_{B_{\rho_n}}|\nabla u|^{q+\alpha_n}\,dx+1 \right\}^{{2^{*}}/2},
\end{align} where $C=C(R_0,\alpha_0,N,M,m,p,q)$ is independent of $n$. Defining $Y_n:=\int_{B_{\rho_{n}}}|\nabla u|^{\alpha_{n}+q}\,dx$, estimate \cref{eqn6.3} becomes
\begin{align}\label{iteration}
	Y_{n+1}\leq C^n (Y_n+1)^{\frac{{2^*}}{2\hfill}}.
\end{align}

By iterating \cref{iteration}, we get
\begin{equation*}\begin{array}{rcl}
	Y_{n+1}&\leq& C^n \left[ Y_n + 1 \right]^{{2^{*}}/2}\\
	&\leq& C^n \left\{ C^{n-1} \left[ Y_{n-1} + 1\right]^{{2^{*}}/2}+1 \right\}^{{2^{*}}/2}\\
	&\leq & C^{n+(n-1)\frac{{2^*}}{2}}2^{\frac{2^*}{2}}\left[Y_{n-1}+1\right]^{\left({{2^*}}/{2}\right)^2}\\
	&\vdots& \\
	&\leq& C^{\sum_{j=0}^n\left({{2^*}}/{2}\right)^j(n-j)}2^{\sum_{j=0}^n\left({{2^*}}/{2}\right)^{j+1}}\left[Y_0+1\right]^{\left({{2^*}}/{2}\right)^{n+1}}.
	\end{array}
\end{equation*}
As a result, we get
\begin{align}\label{finaliter}
	Y_n^{\frac{1}{\alpha_n+q}}\leq \left(C^{\sum_{j=0}^n\left(\frac{{2^*}}{2}\right)^j(n-j)}\right)^{\frac{1}{\alpha_n+q}}\left(2^{\sum_{j=0}^n\left(\frac{{2^*}}{2}\right)^{j+1}}\right)^{\frac{1}{\alpha_n+q}} \left[Y_0+1\right]^{\frac{{(2^*)}^n}{2^n(\alpha_n+q)}}.
\end{align}
We make the following three observations:
\begin{equation}\label{threeobs} \begin{array}{rcl}
	\frac{\left(\frac{{2^*}}{2}\right)^n}{\alpha_n+q}& =& \frac{\left(\frac{{2^*}}{2}\right)^n}{\left(\frac{{2^*}}{2}\right)^n\alpha_0+\left(\frac{{2^*} p-2q}{{2^*}-2}\right)\left(\left(\frac{{2^*}}{2}\right)^n-1\right)+q} \overset{n \nearrow \infty}{\longrightarrow} \frac{1}{\alpha_0+\left(\frac{{2^*} p-2q}{{2^*}-2}\right)}, \\
	\frac{\sum_{j=0}^{n-1}\left(\frac{{2^*}}{2}\right)^j(n-j)}{\alpha_n+q}& \leq & \frac{\left(\frac{\frac{{2^*}}{2}}{\left(\frac{{2^*}}{2}-1\right)^2}\right)\left(\left(\frac{{2^*}}{2}\right)^n-1\right)}{\left(\frac{{2^*}}{2}\right)^n\alpha_0+\left(\frac{{2^*} p-2q}{{2^*}-2}\right)\left(\left(\frac{{2^*}}{2}\right)^n-1\right)+q} \leq   C(N,\al_0,p,q) <\infty, \\
	\frac{\sum_{j=0}^{n-1}\left(\frac{{2^*}}{2}\right)^{j+1}}{\alpha_n+q}& \leq &  \frac{\left(\frac{{2^*}}{2}\right)\sum_{j=0}^{n-1}\left(\frac{{2^*}}{2}\right)^j(n-j)}{\alpha_n+q} \leq C(\al_0,N,p,q) <\infty.
	\end{array}
\end{equation} 
Hence, following the observations from \cref{threeobs} and passing to the limit as $n\to\infty$ in \cref{finaliter}, we get
\begin{align}\label{Lip1}
	||\nabla u||_{L^\infty(B_{R_0/2})}\leq C\left(\int_{B_{R_0}}|\nabla u|^{\alpha_0+q}\,dx+1\right)^{\frac{1}{\alpha_0+\left(\frac{{2^*} p-2q}{{2^*}-2}\right)}},
\end{align} where $C=C(R_0,\al_0,M,m,p,q,N)$.

Now, we make use of \cref{highinteg} with the choice of  $\beta=\alpha_0+q-p \geq 2$, estimate \cref{Lip1} becomes
\begin{align*}
	||\nabla u||_{L^\infty(B_{R_0/2})}\leq C\left(\Gamma_1+1\right)^{\frac{1}{\alpha_0+\left(\frac{{2^*} p-2q}{{2^*}-2}\right)}},
\end{align*} where
\begin{align*}
	\Gamma_1=R_0^N\left[\left(\frac{||U||_{L^\infty(2B)}}{R_0}\right)^{\frac{2(\alpha_0+q)}{p}}+\left(\frac{||U||_{L^\infty(2B)}}{R_0}\right)^{\frac{2(\alpha_0+q)}{2-q+p}}+\left(\frac{||U||_{L^\infty(2B)}}{R_0}\right)^{\alpha_0+q}\right].
\end{align*}
This completes the proof of the theorem.
\end{proof}
\section{Proof of \cref{maintheorem}}\label{section6}

	It remains to obtain the Lipschitz bound for $U$ from the Lipschitz bound \cref{lipbound1} for the regularized minimizer (recall \cref{remark_supression}). We follow the scheme of the proof in \cite{bellaRegularityMinimizersScalar2020}, which is similar to the double approximation procedure in \cite{espositoHigherIntegrabilityMinimizers1999}. Observe that
	\begin{align*}
	||\nabla u_{\sigma,\ve}||_{L^\infty(B_{R_0/2})}&\overset{\cref{lipbound1}}{\leq} C,
	\end{align*} where $C=C(p,q,N,M,m,||U||_{L^\infty(2B)})$. One can also obtain
	\begin{align*}
	m\int_B|\nabla u_{\sigma,\ve}|^p\,dx&\overset{\cref{hyp1}}{\leq}\int_{B}f(\nabla u_{\sigma,\ve})\,dx\leq\int_{B}f_{\sigma}(\nabla u_{\sigma,\ve})\,dx\overset{\text{\cref{mainprob}}}{\leq}\int_{B}f_{\sigma}(\nabla U_{\ve})\,dx\\
	&\overset{\cref{mainprob2}}{=}\int_{B}f(\nabla U_{\ve})\,dx+\frac{\sigma}{2}\int_{B}|\nabla U_{\epsilon}|^2\,dx\leq \int_{(1+\ve)B}f(\nabla U)\,dx+\frac{\sigma}{2}\int_{B}|\nabla U_{\epsilon}|^2\,dx,
\end{align*} where the last inequality follows from the convexity of $f$ and Jensen's inequality. Now, for fixed $\ve>0$, we can find $w_{\ve}\in U_{\ve}+ W^{1,p}_0(B)$ such that, for a subsequence, as $\sigma\to 0$, we have
$$\nabla u_{\sigma,\ve}\overset{*}{\rightharpoonup}\nabla w_{\ve}\mbox{ weak-$*$ in }L^\infty(B_{R_0/2}),$$
$$u_{\sigma,\ve}\rightharpoonup w_{\ve}\mbox{ in }W^{1,p}(B)\mbox{-weak}.$$ Passing to the limit, as $\sigma\to 0$, we obtain, on account of weak and weak-$*$ lower semicontinuity of norms,
\begin{align*}
		||\nabla w_{\ve}||_{L^\infty(B_{R_0/2})}&\leq C, \mbox{ and }
\end{align*}
\begin{align*}
	m\int_{B}|\nabla w_{\ve}|^p\,dx&\leq \int_{B}f(\nabla w_{\ve})\,dx\leq \int_{(1+\ve)B}f(\nabla U)\,dx.
\end{align*}Once again, using the fact that, for a subsequence, $w_{\ve}\rightharpoonup w$ in $U+W^{1,p}_0(B)$-weak, we obtain by lower semicontinuity,
\begin{align}\label{three}
	||\nabla w||_{L^\infty(B_{R_0/2})}&\leq C,\mbox{ and }
\end{align}
\begin{align}\label{four}
	m\int_{B}|\nabla w|^p\,dx&\leq \int_{B}f(\nabla w)\,dx\leq \int_{B}f(\nabla U)\,dx.
\end{align} 
Finally, by the strict convexity of $f$ (see~\cref{strictconvexityoff}), \cref{four} and the fact that $w \in U+W^{1,p}_0(B)$, we have $w=U$. Hence, the Lipschitz continuity of $U$ follows from \cref{three}.

\begin{remark}\label{strictconvexityoff}
    In order for uniqueness to hold, we want to make use of \cref{uniqueness} which requires strict convexity. But in our situation, we have the additional condition \cref{hyp3} and this implies strict convexity of the functional, as can be seen in the calculation from  \cite[Proof of Theorem 4.10]{beckLipschitzBoundsNonuniform2020}. 
\end{remark}


\bibliography{bibfile}

\begin{thebibliography}{10}

\bibitem{baroniRegularityGeneralFunctionals2018}
Paolo Baroni, Maria Colombo, and Giuseppe Mingione.
\newblock Regularity for general functionals with double phase.
\newblock {\em Calculus of Variations and Partial Differential Equations},
  57(2):Paper No. 62, 48, 2018.

\bibitem{beckLipschitzBoundsNonuniform2020}
Lisa Beck and Giuseppe Mingione.
\newblock Lipschitz {{Bounds}} and {{Nonuniform Ellipticity}}.
\newblock {\em Communications on Pure and Applied Mathematics},
  73(5):944--1034, 2020.

\bibitem{bellaRegularityMinimizersScalar2020}
Peter Bella and Mathias Sch\"affner.
\newblock On the regularity of minimizers for scalar integral functionals with
  $(p,q)$-growth.
\newblock {\em Analysis \& PDE}, 13(7):2241–2257, 2020.

\bibitem{BF2002}
M.~Bildhauer and M.~Fuchs.
\newblock Interior regularity for free and constrained local minimizers of
  variational integrals under general growth and ellipticity conditions.
\newblock {\em Zap. Nauchn. Sem. S.-Peterburg. Otdel. Mat. Inst. Steklov.
  (POMI)}, 288(Kraev. Zadachi Mat. Fiz. i Smezh. Vopr. Teor. Funkts.
  32):79--99, 271--272, 2002.

\bibitem{bousquetLipschitzRegularityOrthotropic2020}
Pierre Bousquet and Lorenzo Brasco.
\newblock Lipschitz regularity for orthotropic functionals with nonstandard
  growth conditions.
\newblock {\em Revista Matem\'atica Iberoamericana}, 36(7):1989--2032, 2020.

\bibitem{choeInteriorBehaviourMinimizers1992}
Hi~Jun Choe.
\newblock Interior behaviour of minimizers for certain functionals with
  nonstandard growth.
\newblock {\em Nonlinear Analysis: Theory, Methods \& Applications},
  19(10):933--945, November 1992.

\bibitem{colomboBoundedMinimisersDouble2015}
Maria Colombo and Giuseppe Mingione.
\newblock Bounded minimisers of double phase variational integrals.
\newblock {\em Archive for Rational Mechanics and Analysis}, 218(1):219--273,
  2015.

\bibitem{colomboRegularityDoublePhase2015}
Maria Colombo and Giuseppe Mingione.
\newblock Regularity for double phase variational problems.
\newblock {\em Archive for Rational Mechanics and Analysis}, 215(2):443--496,
  2015.

\bibitem{cupiniExistenceRegularityElliptic2014}
Giovanni Cupini, Paolo Marcellini, and Elvira Mascolo.
\newblock Existence and regularity for elliptic equations under $p,q$-growth.
\newblock {\em Advances in Differential Equations}, 19(7-8):693–724, 2014.

\bibitem{cupiniLocalBoundednessMinimizers2015}
Giovanni Cupini, Paolo Marcellini, and Elvira Mascolo.
\newblock Local boundedness of minimizers with limit growth conditions.
\newblock {\em Journal of Optimization Theory and Applications}, 166(1):1--22,
  2015.

\bibitem{cupiniRegularityMinimizersLimit2017}
Giovanni Cupini, Paolo Marcellini, and Elvira Mascolo.
\newblock Regularity of minimizers under limit growth conditions.
\newblock {\em Nonlinear Analysis: Theory, Methods \& Applications},
  153:294--310, April 2017.

\bibitem{defilippisLipschitzBoundsNonautonomous2020}
Cristiana De~Filippis and Giuseppe Mingione.
\newblock Lipschitz bounds and nonautonomous integrals.
\newblock {\em arXiv:2007.07469 [math]}, July 2020.

\bibitem{defilippisRegularityMinimaNonautonomous2020}
Cristiana De~Filippis and Giuseppe Mingione.
\newblock On the regularity of minima of non-autonomous functionals.
\newblock {\em Journal of Geometric Analysis}, 30(2):1584--1626, 2020.

\bibitem{defilippisInterpolativeGapBounds2021}
Cristiana De~Filippis and Giuseppe Mingione.
\newblock Interpolative gap bounds for nonautonomous integrals.
\newblock {\em Anal. Math. Phys.}, 11(3):Paper No. 117, 39, 2021.

\bibitem{defilippisRegularityMultiphaseVariational2019}
Cristiana De~Filippis and Jehan Oh.
\newblock Regularity for multi-phase variational problems.
\newblock {\em Journal of Differential Equations}, 267(3):1631--1670, 2019.

\bibitem{eleuteriRegularityScalarIntegrals2020}
Michela Eleuteri, Paolo Marcellini, and Elvira Mascolo.
\newblock Regularity for scalar integrals without structure conditions.
\newblock {\em Advances in Calculus of Variations}, 13(3):279--300, July 2020.

\bibitem{espositoHigherIntegrabilityMinimizers1999}
Luca Esposito, Francesco Leonetti, and Giuseppe Mingione.
\newblock Higher integrability for minimizers of integral functionals with
  $(p,q)$ growth.
\newblock {\em Journal of Differential Equations}, 157(2):414–438, 1999.

\bibitem{espositoRegularityMinimizersFunctionals1999}
Luca Esposito, Francesco Leonetti, and Giuseppe Mingione.
\newblock Regularity for minimizers of functionals with $p$-$q$ growth.
\newblock {\em NoDEA. Nonlinear Differential Equations and Applications},
  6(2):133–148, 1999.

\bibitem{espositoRegularityResultsMinimizers2002}
Luca Esposito, Francesco Leonetti, and Giuseppe Mingione.
\newblock Regularity results for minimizers of irregular integrals with $(p,q)$
  growth.
\newblock {\em Forum Mathematicum}, 14(2):245–272, 2002.

\bibitem{espositoSharpRegularityFunctionals2004}
Luca Esposito, Francesco Leonetti, and Giuseppe Mingione.
\newblock Sharp regularity for functionals with $(p,q)$ growth.
\newblock {\em Journal of Differential Equations}, 204(1):5–55, 2004.

\bibitem{giaquintaGrowthConditionsRegularity1987}
Mariano Giaquinta.
\newblock Growth conditions and regularity, a counterexample.
\newblock {\em manuscripta mathematica}, 59(2):245--248, June 1987.

\bibitem{giustiDirectMethodsCalculus2003}
Enrico Giusti.
\newblock {\em Direct Methods in the Calculus of Variations}.
\newblock {World Scientific Publishing Co., Inc., River Edge, NJ}, 2003.

\bibitem{hartman1966bounded}
Philip Hartman.
\newblock On the bounded slope condition.
\newblock {\em Pacific Journal of Mathematics}, 18(3):495--511, 1966.

\bibitem{hirschGrowthConditionsRegularity2020}
Jonas Hirsch and Mathias Sch{\"a}ffner.
\newblock Growth conditions and regularity, an optimal local boundedness
  result.
\newblock {\em Communications in Contemporary Mathematics}, page 2050029, June
  2020.

\bibitem{hongRemarksMinimizersVariational1992}
Min~Chun Hong.
\newblock Some remarks on the minimizers of variational integrals with
  nonstandard growth conditions.
\newblock {\em Unione Matematica Italiana. Bollettino. A. Serie VII},
  6(1):91--101, 1992.

\bibitem{marcelliniExempleSolutionDiscontinue1987}
Paolo Marcellini.
\newblock Un exemple de solution discontinue d'un probleme variationnel dans le
  cas scalaire.
\newblock {\em Mat. Univ. Firenze.}, Preprint No. 11 dell'Ist, 1987.

\bibitem{marcelliniRegularityMinimizersIntegrals1989}
Paolo Marcellini.
\newblock Regularity of minimizers of integrals of the calculus of variations
  with nonstandard growth conditions.
\newblock {\em Archive for Rational Mechanics and Analysis}, 105(3):267--284,
  1989.

\bibitem{marcelliniRegularityExistenceSolutions1991}
Paolo Marcellini.
\newblock Regularity and existence of solutions of elliptic equations with
  $p,q$-growth conditions.
\newblock {\em Journal of Differential Equations}, 90(1):1–30, 1991.

\bibitem{marcelliniRegularityEllipticEquations1993}
Paolo Marcellini.
\newblock Regularity for {{Elliptic Equations}} with {{General Growth
  Conditions}}.
\newblock {\em Journal of Differential Equations}, 105(2):296--333, October
  1993.

\bibitem{marcelliniEverywhereRegularityClass1996}
Paolo Marcellini.
\newblock {Everywhere regularity for a class of elliptic systems without growth
  conditions}.
\newblock {\em Annali della Scuola Normale Superiore di Pisa - Classe di
  Scienze}, 23(1):1--25, 1996.

\bibitem{marcelliniRegularityScalarVariational1996}
Paolo Marcellini.
\newblock Regularity for some scalar variational problems under general growth
  conditions.
\newblock {\em Journal of Optimization Theory and Applications},
  90(1):161--181, July 1996.

\bibitem{marcelliniRegularityGeneralGrowth2020}
Paolo Marcellini.
\newblock Regularity under general and $p,q$-growth conditions.
\newblock {\em Discrete and Continuous Dynamical Systems. Series S},
  13(7):2009–2031, 2020.

\bibitem{mingioneRecentDevelopmentsProblems2021}
Giuseppe Mingione and Vicen{\c t}iu R{\u a}dulescu.
\newblock Recent developments in problems with nonstandard growth and
  nonuniform ellipticity.
\newblock {\em Journal of Mathematical Analysis and Applications}, page 125197,
  March 2021.

\bibitem{mirandaTeoremaDiEsistenza1965}
Mario Miranda.
\newblock Un teorema di esistenza e unicità per il problema dell'area minima
  in $n$ variabili.
\newblock {\em Annali della Scuola Normale Superiore di Pisa. Classe di
  Scienze. Serie III}, 19:233–249, 1965.

\bibitem{rindler}
Filip Rindler.
\newblock {\em Calculus of variations}.
\newblock Springer, 2018.

\bibitem{simon2011}
Barry Simon.
\newblock {\em Convexity: an analytic viewpoint}, volume 187.
\newblock Cambridge University Press, 2011.

\bibitem{stampacchiaRegularMultipleIntegral1963}
Guido Stampacchia.
\newblock On some regular multiple integral problems in the calculus of
  variations.
\newblock {\em Communications on Pure and Applied Mathematics}, 16:383--421,
  1963.

\end{thebibliography}
\bibliographystyle{plain}

\end{document}